\documentclass[12pt]{elsarticle}
\usepackage{amsfonts,amsmath,amsthm,amssymb}

\usepackage[active]{srcltx}

\newtheorem{theorem}{Theorem}
\newtheorem{lemma}{Lemma}

\DeclareMathOperator{\rank}{rank}

\begin{document}

\title{Commutative Bezout domains of stable range 1.5\\
{\normalsize \it Dedicated to the 70-th birthday  of Professor V.V.~Sergeichuk}}

\author[bov]{Victor A. Bovdi\corref{cor}} \ead{vbovdi@gmail.com}
\address[bov]{United Arab Emirates University, Al Ain, UAE}

\author[sh]{Volodymyr P.~ Shchedryk} \ead{shchedrykv@ukr.net}
\address[sh]{Pidstryhach Institute for Applied Problems of Mechanics
and Mathematics,  National Academy of Sciences of Ukraine, Lviv,  Ukraine}

\cortext[cor]{Corresponding author.}

\begin{abstract}
A ring $R$ is said to be of stable range 1.5  if for each  $a, b\in R$ and $0\neq c\in R$  satisfying  $aR + bR + cR = R$ there exists   $r \in R$ such that $(a + br)R + cR = R$.  Let $R$ be a commutative domain in which  all   finitely generated ideals are principal, and let $R$ be of stable range 1.5. Then each matrix $A$ over $R$ is reduced to Smith's canonical form by transformations $PAQ$ in which $P$ and $Q$ are invertible and at least one of them  can be chosen  to be a product of elementary matrices. We generalize Helmer's theorem about the   greatest common divisor  of entries of $A$  over $R$.
\end{abstract}


\maketitle

\section{Introduction and main  results}

The problem of finding  canonical forms of a matrix  up  to equivalency  is classical. The rings over which matrices are equivalent to certain diagonal matrices have been studied extensively. In the present paper we investigate such question for some classes of  commutative Bezout domains.

The matrix  ${\rm diag}(d_1, d_2,\ldots)$ means  a (possibly rectangular) matrix
having $d_1, d_2,\ldots $ on  main diagonal and zeros elsewhere (by the main
diagonal we mean the one beginning at the upper left corner). We use the following  notations  of commutative rings: $(a_1,\ldots, a_n)$ denotes   the greatest common divisor  of elements $a_1,\ldots, a_n$ and   $a|b$  means that  $a$ is a divisor of  $b$. The set of all matrices of size $n\times m$ over a
 ring $R$ is denoted by $R^{n\times m}$.

An associative (not necessary commutative) ring $R$ is called an  elementary divisor ring (introduced by  I.~Kaplansky in \cite{Kaplansky})  if every (not necessary square)  matrix $A$  over $R$  admits a diagonal reduction, that is, there exist   invertible matrices $P$ and $Q$ over the ring $R$ such that
\begin{equation}\label{EEE:1}
PAQ={\rm diag}(\varphi_1, \; \ldots , \; \varphi_k, \; 0, \; \ldots , \; 0) =\Phi,
\end{equation}
in which  each element $\varphi_i$ is a total divisor of $\varphi_{i+1}$ for $ i=1, \ldots , k-1$ (i.e.
\[
R\varphi_{i+1}R\subseteq \varphi_i R\cap R\varphi_i,
\]
which  is equivalent to $\varphi_{i}| \varphi_{i+1}$ when  $R$ is commutative).
The matrix $\Phi$ is called the {\it Smith normal form}  and
$\varphi_1,\ldots, \varphi_k$ are the   {\it  invariant factors} of the matrix $A$. Examples  of such rings are the ring of integers $\mathbb{Z}$ (see \cite{Smith}), Euclidean rings and  principal ideal rings  (see \cite{Jacobson, Wedderburn}).

A ring $R$ is a  {\it Bezout} ring if  each of its   finitely generated ideals is  principal.
Each matrix from  $R^{1\times n}$ and   $R^{n\times 1}$ over an  elementary divisor ring $R$ admits a diagonal reduction. This is equivalent to the condition that  each   finitely generated ideal in $R$ is principal. Hence an  elementary divisor ring  is a Bezout ring.
Gilman and  Henriksen constructed an example of a commutative Bezout ring which is not an  elementary divisor ring (see \cite[Example 4.11, p.\,382]{Gillman_Henriksen_2}). This raises the problem  whether  an arbitrary  commutative Bezout domain is an  elementary divisor ring. Euclidean rings and principal ideal rings  satisfy the ascending  chain condition on ideals.
However  Helmer  \cite{Helmer} showed that this condition  can be replaced by the less restrictive hypothesis that $R$ is   adequate.

A commutative Bezout domain $R$ is  {\it adequate} if  for  $a,b\in R$ with  $a\not=0$,  there exist $r,s\in R$ such that $a=rs$,  in which   $(r,b)=1$ and if $s'$ is a non-unit divisor of $s$, then  $ (s',b)\neq 1$.  Commutative principal ideal domains and commutative regular rings with identity  are adequate rings; see \cite[Theorem 11, p.\,365]{Gillman_Henriksen} (see also \cite{Chen_Sheibani}).

The proof of the fact that  an adequate ring is an  elementary divisor ring (see \cite[Theorem 3, p.\,234]{Helmer})  was based on  \cite[Theorem 1, p.\,228]{Helmer} which says that  if   $A\in R^{n \times m}$ has  maximal rank over   an adequate ring $R$, then  there is a row   $u =[1, {u_{2}}, {\ldots},   {u_{n}}]\in R^{1\times n}$ such that  the    ${\rm g.c.d.}$ of the entries of $uA$  and the  ${\rm g.c.d.}$ of entries of the matrix $A$  coincide.  It means that there is an invertible matrix
$
U= {\left[\begin{smallmatrix}
  1 & u_2 & \ldots & u_{n} \\ 0 & 1 &
 \ldots & 0 \\ \vdots  & & \ddots & \\ 0 & 0 & \ldots & 1
\end{smallmatrix}\right] }$
such that ${\rm g.c.d.}$ of entries of the matrix $A$ and ${\rm g.c.d.}$ of the elements of the first  row of matrix   $UA$ coincide. This result was generalized for matrices with rank greater then one by Petrychkovych in \cite[Lemma 3.1,p.\,71] {Petrichkovich}.

Deep studies of the theory of   elementary divisor rings increasingly suggest that   methods of  pure  ring theory are insufficient.  Promising  studies were based on the concept of  stable range of rings, introduced by  Bass \cite{Bass_book} as  an  important $K$-theory invariant.

According to \cite[Property $(7.2)_n$, p.\,106]{Bass_Milnor_Serre} (see also   the definition after Lemma 1 in \cite{Vaserstein}),   the {\it stable range} of a ring $R$ is  the smallest positive integer  $n$ such that the following condition holds:

\qquad $(*)_n$\quad  for each  $a_1, \ldots, a_{n + 1} \in R$ satisfying
$a_1R + \cdots + a_{n+1}R = R$
there exist  $b_1, \ldots, b_n \in R$ for which
\[
(a_1 +a_{n+1}b_1)R+ \cdots +(a_n +a_{n+1}b_n)R = R.
\]
If such $n$ does not exist, then the  stable range of $R$ is infinity.

The concept of the  stable range  of a ring turned out to be useful in the study of elementary divisor rings.
In particular, Zabavsky \cite[Theorem 1, p.\,666]{Zabavsky_2} proves that each   elementary divisor ring has stable range $\le 2$. His survey  \cite{Zabavsky_survey} contains results on the problem when  a  commutative Bezout domain is an  elementary divisor ring.

We say that associative  ring $R$ has  \emph{stable range 1.5}  if for each  $a, b\in R$ and $0\neq c\in R$  satisfying  $aR + bR + cR = R$ there exists   $r \in R$ with
\[
(a + br)R + cR = R.
\]
This notion  was  introduced by the second author  \cite{Shchedryk_3} and studied in \cite{Shchedrik_4, Shchedryk}. Commutative principal ideal domains, adequate rings (see   \cite[Prorositions 3.15 and 3.14]{McGovern} and  \cite[Proposition 4]{Anderson_Juett}),
rings of $2\times 2$ matrices over rings listed before  (see \cite[Theorem 5, p.\,856]{Shchedryk_3}) has stable range 1.5.

Evidently each ring  with stable range $1.5$ has Bass stable range $2$. The converse is not always true. For instance, the subring $\mathbb{Z} +x \mathbb{Q}[[x]]$ of the ring of formal power series  $\mathbb{Q}[[x]]$ over the field of rational numbers $\mathbb{Q}$ (see \cite[Example 1, p.\,160]{Henriksen}) has stable range $2$ but not $1.5$ (see \cite[ Example 1.1, p.\,22]{Shchedryk_book}).  This shows that the rings of stable range 1.5 are between the rings of stable range 1 and 2, respectively.

Note that the notion of  stable range 1.5 is closely related to the concept  of almost stable range 1, introduced by McGovern \cite{McGovern}.
A ring $R$ has {\it  almost  stable range} 1 if each proper homomorphic image of $R$ has
stable range 1.
In such rings if $aR + bR + cR = R$,  where $c$ does not belong to the Jacobson radical  ${\mathfrak{J}}(R)$ of $R$,  then  there exists   $r \in R$  (see \cite[Theorem 3.6]{McGovern}) such that
\[
(a + br)R + cR = R.
\]
In commutative rings of almost stable range 1 the condition $c \notin {\mathfrak{J}}(R)$ can be replaced to  $c\neq 0 $  (see \cite[Proposition 4]{Anderson_Juett}). Using this result and \cite[Theorem 3.7]{McGovern} we conclude that each   commutative Bezout domain of stable range  1.5 is an  elementary divisor ring.

Several authors define and study rings of  idempotent stable range \cite{Chen, McGovern}, rings of unit stable range  \cite{Goodearl_Menal},  rings of  neat stable range  \cite{Zabavsky_3}, and rings of square stable range \cite{Khurana_Lam_Wang}.

Our first  result is a  generalization of      Helmer's result  \cite[Theorem 1]{Helmer}.

\begin{theorem}\label{Th:1}
If $R$  is a commutative Bezout domain, then
the following conditions are equivalent:
\begin{itemize}
\item[(i)] $R$ has stable range $1.5;$

\item[(ii)] for each  $A\in R^{n \times m}$  with ${\rm rank} (A)>1$, there exists
 $u =[1,  {u_{2}},   {\ldots},  {u_{n}}]\in R^{1\times n}$
such that
$uA = [{b_{1}}, {b_{2}},  {\ldots},  {b_{m}}]$,
in which  $(b_{1},   b_{2}, \ldots, b_{m} )$ coincides with the  {\rm g.c.d.} of entries of the matrix   $A$.
\end{itemize}
 \end{theorem}

The following example shows that the condition $\rank(A)>1$ in  Theorem $1(ii)$ is essential.
Indeed, let $A:=\left[ \begin{smallmatrix}
5& 0\\
7& 0
\end{smallmatrix}
\right]\in \mathbb{Z}^{2\times 2}$ and  $u=[1,u_2]\in \mathbb{Z}^{1\times 2}$. Then
$uA=[5+7u_2, 0]$. The  {\rm g.c.d.} of entries of matrix $A$ is equal to 1. But
\[
(5+7u_2, 0)=5+7u_2\not=\pm 1
\]
 for any $u_2\in\mathbb{Z}$.
However the ring $\mathbb{Z}$ has stable range $1.5$.

%This implies  that the prescribed   row   $u=[1,u_2]$   does not exist

Matrices $P$ and $Q$ satisfying  equality (\ref{EEE:1}) are called {\it transforming matrices} of the matrix $A$.  The set of all transforming matrices $P$ and $Q$ are denoted by $\mathrm{T}_l(A)$ and  $\mathrm{T}_r(A)$, respectively.  An {\it elementary matrix} is a matrix which is obtained from the identity matrix by elementary  transformations.

\begin{theorem}\label{Th:2}
Let $R$  be a commutative Bezout domain of stable range  $1.5$. If
$A=[a_{ij}] \in R^{n\times m}$ with    ${\rm rank}(A)>1$, then  both of the sets  $\mathrm{T}_l(A)$,   $\mathrm{T}_r(A)$  contain an elementary matrices.
\end{theorem}

\section{Proofs}

If  $A, B$ are   matrices such that $A=UBV$ for some  invertible matrices $U$ and $V$, then we say $A\sim B$. We use the following result proved in \cite[Property 6, p.\,50]{Shchedryk_3}.

\begin{lemma}\label{L:1}
Let $R$ be  a commutative Bezout domain of stable range  1.5. Let  $a_1, \ldots, a_n$ be a collection of relatively prime elements in $R$  and $0\ne\psi\in R$. Then there exist $u_1, \ldots, u_n\in R$    such that
\begin{itemize}
\item[\rm (i)] $ u_1a_1 + \cdots + u_na_n = 1$;
\item[\rm (ii)] $ (u_1, \ldots, u_k )= (\psi, u_k)=1$ for each  fixed $2 \leq k\leq n$.
\end{itemize}
\end{lemma}

\begin{proof}[Proof of Theorem \ref{Th:1}]  $(i) \Longrightarrow (ii)$.  Let   $A\in R^{n \times m}$  has   rank greater than 1. Without loss of generality, assume that   $m \geq n$. A ring $R$ is an  elementary divisor ring. Therefore the equation  \eqref{EEE:1} holds for some invertible  $P$ and  $Q$.  Since $\rank(A)> 1$, we have $k\geq 2$ in  \eqref{EEE:1}. Consider
\[
U:=\left[   \begin{smallmatrix}
0 & 0 & 1 & 0 \\
0 & I_{k-2} & 0 & 0 \\
1& 0 & 0 & 0 \\
0 & 0 & 0 & I_{n-k} \\
 \end{smallmatrix}\right]\quad  \text{and}\quad  V:=\left[   \begin{smallmatrix}
0 & 0 & 1 & 0 \\
0 & I_{k-2} & 0 & 0 \\
1& 0 & 0 & 0 \\
0 & 0 & 0 & I_{m-k} \\
 \end{smallmatrix}\right]
\]
in which $I_s$ ($s\geq 1$) is the  identity $s\times s$ matrix and $I_0$ is an empty matrix.
It is easy to check that
\[
 \Phi':= (UP)A(QV)={\rm diag}(\varphi_k, \; \varphi_2,\;
  \varphi_3,\; \ldots , \; \varphi_{k-1}, \;\varphi_1, \;0,\ldots,0).
\]
If  $P_1:=(UP)\det(UP)^{-1}$ and  $Q_1:=(QV)\det(UP)$, then  $P_1AQ_1=\Phi'$ and from     $\det(P_1^{-1})=1$ we obtain that $\det(P_1^{-1})=\sum_{i=1}^n(-1)^{i+1}p_{1i}\Delta_{i}=1$,
where $P_1^{-1}:=[p_{ij}]$ and $\Delta_{1}, \ldots, \Delta_{n} $ are the  corresponding minors.   There exist   $s_{11}, \ldots , s_{1n}$  by Lemma \ref{L:1}, such that
$\sum_{i=1}^n  s_{1i}\Delta_{i} = 1$ and
\begin{equation}\label{EEE:2}
 (s_{11}, s_{12}, \ldots , s_{1k})= (\varphi_k, s_{1k})=1.
\end{equation}
Moreover (see \cite[Property 3, p.\,48]{Shchedryk_3})
\begin{equation}\label{EEE:3}
 [s_{11}, s_{12},  \ldots,   s_{1n}]=[1,  t_2,  \ldots,  t_{n}] P_1^{-1},
\end{equation}
for some   $t_2, \ldots ,  t_{n} \in R$.
The equality $P_1AQ_1=\Phi'$ implies  $AQ_1=P_1^{-1}\Phi'$, and  so by  \eqref{EEE:3} we have
\begin{equation}\label{EEE:4}
\begin{split}
 {\left[\begin{smallmatrix}
  1 & t_2 & \ldots & t_{n} \\
  0 & 1 & \ldots & 0 \\
 \vdots  & & \ddots & \\
 0 & 0 & \ldots & 1
\end{smallmatrix} \right]}
AQ_1&=
 {\left[\begin{smallmatrix}
 s_{11} &  s_{12} & \ldots &  s_{1n} \\
   p_{21} &  p_{22} & \ldots &  p_{2n}\\
  \dots &\dots & \dots &\dots \\
   p_{n1} &  p_{n2} & \ldots &  p_{nn}
\end{smallmatrix} \right]} \Phi'\\
% Novaja stroka
&= \left[\begin{smallmatrix}
 s_{11}\varphi_k &  s_{12}\varphi_2 & \ldots  &   s_{1.k-1}\varphi_{k-1}
 & s_{1k}\varphi_1 & 0 &  {\ldots} & 0  \\
  p_{21}\varphi_k &  p_{22}\varphi_2 & \ldots  &   p_{2.k-1}\varphi_{k-1}
  & p_{2k}\varphi_1 & 0 &  {\ldots} & 0 \\
 \ldots &\ldots & \ldots & \ldots & \ldots & \ldots &\ldots &\ldots \\
  p_{n1}\varphi_k &  p_{n2}\varphi_2 & \ldots  &   p_{n.k-1}\varphi_{k-1}
   & p_{nk}\varphi_1 & 0 &  {\ldots} & 0
\end{smallmatrix}\right].
\end{split}
\end{equation}
According to \eqref{EEE:2},  the  greatest common divisor $\tau$ of elements of the  first row   of the last  matrix in  \eqref{EEE:4} is equal to
\[
\begin{split}
\tau\;=&\; (s_{11}\varphi_k,  s_{12}\varphi_2, \ldots  ,   s_{1.k-1}\varphi_{k-1}, s_{1k}\varphi_1 )\\
=&\varphi_1\left(s_{11}\frac{\varphi_k}{\varphi_1}, s_{12}\frac{\varphi_2}{\varphi_1},\ldots,
\frac{\varphi_{k-1}}{\varphi_1}s_{1.k-1},   s_{1k} \right)\\
=&\varphi_1\left(\left(s_{1k}, s_{11}\frac{\varphi_k}{\varphi_1}\right),
\left(s_{1k}, s_{12}\frac{\varphi_2}{\varphi_1}\right), \ldots  ,   \left(s_{1k}, s_{1.k-1}\frac{\varphi_{k-1}}{\varphi_1}\right) \right)\\
=&\varphi_1(s_{11}, s_{12}, \ldots , s_{1k})=\varphi_1.
\end{split}
\]
Using \eqref{EEE:4} and reasoning as above,   we obtain that
\[
\begin{split}
(1, {t_{2}},  {\ldots},  t_{n}) AQ_1 &=(s_{11},  s_{12},   {\ldots},  s_{1n})\Phi'\\
&=(s_{11}\varphi_k,   s_{12}\varphi_2, \ldots,    s_{1.k-1}\varphi_{k-1},  s_{1k}\varphi_1,   0,\ldots,0)\\
&\sim (\varphi_1 , 0,   {\ldots},  0 ),
\end{split}
\]
so $(1,  {t_{2}},  {\ldots}, t_{n})A
\sim (\varphi_1, 0, {\ldots},  0)$,  in which $\varphi_1$ is the  ${\rm g.c.d.}$  of entries of  $A$.

\medskip

$(i) \Longleftarrow (ii)$.  Let $A:=\left[ \begin{smallmatrix}
a & c \\
b & 0
\end{smallmatrix}\right]\in R^{2\times 2}$,  where    $a,b\in R$,    $0\neq c\in R$, and $ (a,b,c)=1$.  There is $[1, r]\in R^{1\times 2}$ such that $[1, r]A=[b_1,  b_2]$, where
\[
 (b_1, b_2)= (a,b,c)=1.
\]
Hence $ (a+br,c)=1$ and so  $R$ has  stable range  1.5. \end{proof}

\begin{proof}[Proof of Theorem \ref{Th:2}]   By Theorem \ref{Th:1},   there exists $u =[1, {u_{2}},  {\ldots},  {u_{n}}]\in R^{1\times n}$
such that $uA =[{b_{1}},   {\ldots},  {b_{n}}]$,  in which  $\varphi_1:= (b_{1}, \ldots, b_{m} )$ is equal to the   ${\rm g.c.d.}$  of entries of $A$. Thus  $\varphi_1 $ is the first invariant factor of  $A$ (see \eqref{EEE:1}).  Clearly
\[
U_1A=
\textstyle
\begin{bmatrix}
b_{1} &  b_{2} & \ldots &  b_{m} \\
   a_{21} &  a_{22} & \ldots &  a_{2m}\\
  \dots &\dots & \dots &\dots \\
   a_{n1} &  a_{n2} & \ldots &  a_{nm}
\end{bmatrix}, \quad \text{where } U_1:=\textstyle
\begin{bmatrix}
  1 & u_2 & \ldots & u_{n} \\ 0 & 1 &
 \ldots & 0 \\ \vdots  & & \ddots & \\ 0 & 0 & \ldots & 1
\end{bmatrix}.
\]
There exists an invertible  $V_1$  such that
$U_1AV_1=
 \left[\begin{smallmatrix}
\varphi_{1} &  0 & \ldots &  0 \\
   a'_{21} &  a'_{22} & \ldots &  a'_{2m}\\
  \ldots &\ldots & \ldots &\ldots \\
   a'_{n1} &  a'_{n2} & \ldots &  a'_{nm}
\end{smallmatrix}\right]$.
%\]
Then   $\varphi_{1}$ is also the   ${\rm g.c.d.}$  of entries of $A_1:=U_1AV_1$.  Thus,
  $\varphi_{1}| a'_{i1}$ for  $i= 2, \ldots , n$ and
\[
U'_1A_1=
\textstyle
\begin{bmatrix}
\varphi_{1} & 0 & \ldots &  0 \\
  0 &  c_{22} & \ldots &  c_{2m}\\
  \dots &\dots & \dots &\dots \\
  0 &  c_{n2} & \ldots &  c_{nm}
\end{bmatrix},  \quad \text{where }
 U'_1:=\textstyle
\begin{bmatrix}
  1 & 0 & \ldots & 0 \\
  -\frac{a'_{21}}{\varphi_{1}} & 1 & \ldots & 0 \\
   \vdots  & & \ddots & \\
   -\frac{a'_{n1}}{\varphi_{1}} & 0 & \ldots & 1
\end{bmatrix}.
\]
Note that $U_1$ and $U'_1$ are elementary matrices.

Consider the  submatrix $B:=\left[\begin{smallmatrix}
    c_{22} & \ldots &  c_{2m}\\
   \dots & \dots &\dots \\
    c_{n2} & \ldots &  c_{nm}
\end{smallmatrix}\right]\in R^{(n-1)\times (m-1)}$ of $U'_1A_1$.
Using the same technique as above, we can find  an  elementary matrix $Z$ and an invertible matrix $W$ such that
$ZB_2W=\left[\begin{smallmatrix}
\varphi_{2} & 0 & \ldots &  0 \\
  0 &  d_{33} & \ldots &  d_{3m}\\
  \dots &\dots & \dots &\dots \\
  0 &  d_{n3} & \ldots &  d_{nm}
\end{smallmatrix}\right]$,
in which  $\varphi_{2}$ is the second invariant factor of $A$.

Evidently $U_2:=1\oplus Z$ is elementary,  $V_2:=1\oplus W\in GL_m(R)$ and
\[
(U_2U'_1U_1)A(V_1V_2)={\rm diag} (\varphi_{1}, \varphi_{2})\oplus F, \qquad (F\in R^{(n-2)\times (m-2)}).
\]
Continuing  this process we obtain that  there exist $P\in GL_n(R)$ and $Q\in GL_m(R)$ such that $PAQ={\rm diag} (\varphi_{1},\ldots,  \varphi_{k})$ in which  $P$ is a product of elementary matrices.

Taking $A^T$ instead of $A$ and applying the same reduction, we construct ''new'' matrices $P$ and $Q$ such that $\Phi:=PA^TQ$ is the Smith canonical matrix and $P$  is a product of  elementary matrices. Since $\Phi$ is symmetric, $\Phi=\Phi^T=Q^TAP^T$, where  $P^T$  is a product of  elementary matrices.
\end{proof}

Note that from Theorem \ref{Th:2}  does not imply  that in (\ref{EEE:1}) both of  $P$ and   $Q$ are elementary.

\section*{ Acknowledgement}
The authors are grateful for the referee's valuable  remarks and suggestions.
The research  was supported  by the UAEU UPAR grant G00002160.

\end{document}